\documentclass[a4paper,11pt]{amsart}
\usepackage[T1]{fontenc}
\usepackage[english]{babel}
\usepackage[cp1252]{inputenc}
\usepackage{amsthm}
\usepackage{amsmath}
\usepackage{amsfonts}
\usepackage{stmaryrd}
\usepackage{amssymb}
\usepackage{hyperref}
\usepackage{indentfirst}
\usepackage{amssymb}
\usepackage{eufrak}
\usepackage{mathrsfs}
\usepackage{xypic}
\usepackage[pdftex]{graphicx}

\def\Jac{{\mathop{\rm Jac}}}

\newcommand{\Z}{\mathbb Z}

\newcommand{\Log}{\mathrm{Log}\,}

\newcommand{\Vol}{\textrm{Vol}}

\newtheorem{remark}{Remark}[section]

\newtheorem{theorem}{Theorem}[section]
\newtheorem{definition}{Definition}[section]

\newtheorem{lemma}{Lemma}[section]

\begin{document}
\title{Amoeba finite basis does not exist in general}
\author{Mounir Nisse}
\date{}

\address{School of Mathematics KIAS, 87 Hoegiro Dongdaemun-gu, Seoul 130-722, South Korea.}
\email{\href{mailto:mounir.nisse@gmail.com}{mounir.nisse@gmail.com}}
\thanks{This research  is partially supported by NSF MPS grant DMS-1001615, and Max Planck Institute for Mathematics, Bonn, Germany.}
\subjclass{14T05, 32A60}
\keywords{Amoebas, coamoebas, and critical points of the logarithmic map}


\begin{abstract}
 We show that the amoeba of a generic complex algebraic variety of codimension $1<r<n$ do not have a finite basis. In other words, it is not the intersection of finitely many hypersurface amoebas. Moreover we give a geometric characterization of the topological  boundary of hypersurface amoebas  refining an earlier result of  F. Schroeter  and T. de Wolff \cite{SW-13}.
\end{abstract}

\maketitle



\section{Introduction}

Tropical geometry  combines  aspects of algebraic geometry, discrete geometry, computer algebra,  mirror symmetry and symplectic geometry. This geometry can be seen as a limiting regime of algebraic geometry, where some of its interesting varieties are the limit  of the so-called {\em amoebas}.  Amoebas of  algebraic (or analytic) varieties  are their image under the logarithm with base a real number $t$. In many cases, a tropical variety is the limit of these amoebas as $t$ goes to infinity (e.g., the case of tropical hypersurfaces).
In other words,  tropical objects are some how, the image of  a classical objects under the logarithm with base infinity, they are called non-Archimedean amoebas (this last naming comes from another view of tropical geometry, which coincides with the limiting view in the case of hypersurfaces, see for example \cite{IMS-07}).  

\vspace{0.3cm}

\indent Given an algebraically closed field $K$ endowed with a non-trivial real valuation $\nu : K \rightarrow \mathbb{R} \cup \{ \infty \}$, the tropical variety $\mathcal{T}rop (\mathcal{I})$ of an ideal $\mathcal{I}\subset K[x_1,\ldots, x_n]$ is defined as the topological closure of the set
$$
\nu (V(\mathcal{I})):= \{(\nu (x_1), \ldots , \nu (x_n))\,|\, (x_1,\ldots ,x_n)\in  V(\mathcal{I})\} \subset \mathbb{R}^n,
$$
where $V(\mathcal{I})$ denotes the zero set of $\mathcal{I}$ in $(K^*)^n$ (see for example  \cite{MS-09}).
A tropical basis for $\mathcal{I}$ is a generating set $\mathscr{B} = \{ g_1,\ldots , g_l\}$ of $\mathcal{I}$ such that
$$
 \mathcal{T}rop (\mathcal{I}) = \bigcap_{j=1}^l \mathcal{T}rop (\mathcal{I}_{g_j}),
$$
 where $\mathcal{I}_{g_j}$ denotes the  principal ideal generated by the  polynomial $g_j$. Bogart, Jensen, Speyer, Sturmfels, and Thomas initiated the
computational investigation of tropical bases \cite{BGSS-07} by  providing  Gr$\ddot{o}$bner bases techniques for computing tropical bases as well as by  lower bounds on the size of such bases  when $K$ is the field of Puiseux series $\mathbb{C}((t))$ and the ideal  $\mathcal{I}$ is linear with constant coefficients. Dropping the assumption on the degree of the polynomials,  Hept and Theobald  showed that there always exists a small tropical basis for a prime ideal $\mathcal{I}$ (see \cite{HT-09}). 

\vspace{0.2cm}

We will use  logarithm   with base $e$, so that the Archimedean amoeba of a subvariety of the complex torus $(\mathbb{C}^*)^n$
 is its image under the coordinatewise logarithm map. Amoebas were introduced by Gelfand, Kapranov, and Zelevinsky in 1994 \cite{GKZ-94}. The coamoeba of a subvariety of $(\mathbb{C}^*)^n$ is its image  under the coordinatewise argument map  to the real torus $(S^1)^n$. Coamoebas were introduced by Passare in a talk in 2004 (see e.g.,  \cite{NS-11} for more details about coamoebas).

\vspace{0.2cm} 

A variety $V\subset (\mathbb{C}^*)^n$ of codimension $r$ is  {\em generic} if it contains a point $p$ such that the Jacobian of the logarithmic map restricted to $V$ at the point $p$   has maximal rank  i.e.,  equal to $\min\{ n, 2(n-r)\}$. If the defining ideal $\mathcal{I}(V)$ of $V$ is generated by the set of polynomials  $\{ g_i\}_{i=1}^l$ with the following properties:
\begin{itemize}
\item[(i)]\,  $\mathscr{A}(V) = \bigcap_{i=1}^l \mathscr{A}(V_{g_i})$;
\item[(ii)]\, $\mathscr{A}(V)\subsetneq \bigcap_{i\in \{ 1,\ldots ,l\}\setminus s} \mathscr{A}(V_{g_i})$ for every $1\leq s\leq l$, \end{itemize}
then we said that $\{ g_i\}_{i=1}^l$ is an {\em amoeba basis} of $\mathscr{A}(V)$.

\vspace{0.2cm}

The aim of this paper is to show that  the main  result of Hept and Theobald  in \cite{HT-09} does not have an analogue 
 for  Archimedean amoebas of generic  complex varieties  of positive dimension and not  hypersurfaces. 

\begin{theorem}\label{Amoeba Basis}
If $V$ is a generic complex algebraic variety of codimension  $r$ with $1<r<n$, then its amoeba cannot have a finite basis.
\end{theorem}

 \vspace{0.2cm}

This paper is organized as follows. In Section 2, we prove our main Theorem \ref{Amoeba Basis}, in Section 3 we describe the example of a generic line in the space. In Section 4, we give a geometric characterization of the topological  boundary of hypersurface amoebas and prove Theorem \ref{boundary Amoeba} which refines the main theorem of  F. Schroeter and T. de Wolff  in \cite{SW-13}.

\vspace{0.3cm}


\section{Non  existence of finite amoeba basis in general }

Let $V\subset (\mathbb{C}^*)^n$ be an algebraic variety of codimension $r$ with defining ideal $\mathcal{I}(V)$ and amoeba $\mathscr{A}(V)$. It was shown by Purbhoo \cite{P-08} (a short proof for both amoebas and coamoebas can be found in \cite{NP-11}) that the amoeba $\mathscr{A}(V)$ of $V$ is equal to the intersection of all hypersurface amoebas with defining polynomial in the ideal $\mathcal{I}(V)$, i.e.,
$$
\mathscr{A}(V) = \bigcap_{f\in\mathcal{I}(V)} \mathscr{A}(V_f),
$$
where $V_f$ is the hypersurface with defining polynomial $f$. One naturally ask: Is the amoeba $\mathscr{A}(V)$ the intersection of a finite number of hypersurface amoebas?  In \cite{P-08}, Purbhoo expects a negative answer to this question  in general, but he does not give a formal proof. We
 give a negative answer to this question   when the codimension of our variety is different than
1 and $n$. In \cite{SW-13},  Schroeter and de Wolff give a positive answer to this question when $V$ is the zero-dimensional solution set of a generic  linear system of $n$ equations.
\vspace{0.3cm}

Let us fix some notation and definitions. Let $V\subset (\mathbb{C}^*)^n$ be an algebraic variety of codimension $1<r<n$ with defining ideal $\mathcal{I}(V)$ and amoeba $\mathscr{A}(V)$.
To prove Theorem \ref{Amoeba Basis}, we will deal  with two cases, the one where $\min\{ n, 2(n-r)\} = n$ and the other case  where $\min\{ n, 2(n-r)\} = 2(n-r)$. Namely, the cases where the dimension of the ambient space is less or equal to twice the dimension of $V$, and the case where the dimension of the ambient space is strictly greater than  twice the dimension of $V$. In the first case, the amoeba $\mathscr{A}(V)$ has dimension $n$, and then it  necessarily  has a boundary. In the second case, it may be without boundary, as we will see in some examples.

%
%
%
%
%

\vspace{0.2cm}

Let $V\subset (\mathbb{C}^*)^n$ be a generic algebraic variety of codimension $r$ such that   $n\leq 2(n-r)$. Moreover, assume that the set   of polynomials $\{ g_i\}_{i=1}^l$ is an amoeba basis  of   $\mathscr{A}(V)$.

\vspace{0.2cm}

\noindent {\it Claim} A: With the above hypotheses, let   $x$ be a point in $\partial\mathscr{A}(V)$.  Then   there exists a vector direction $v$  such that $(x+\varepsilon v)\notin \mathscr{A}(V)$ for all small  positive real numbers $\varepsilon$. 

\vspace{0.2cm}

\noindent {\it Proof of Claim} A: First of all, the set of vector direction around $x$ can be identified to the $(n-1)$-dimensional sphere. Assume on the contrary that for all vector direction $v$  and any $\eta >0$ there exists 
$\varepsilon$  with $0<\varepsilon \leq \eta$ and  $(x+\varepsilon v)\in \mathscr{A}(V)$.  
This means that $x\in \mathring{\mathscr{A}}(V_{g_i})$ for all   $1\leq i\leq l$, where $\mathring{\mathscr{A}}(V_{g_i})$ denotes the interior of the hypersurface amoeba $\mathscr{A}(V_{g_i})$. In fact, if there exists  $s$ with $1\leq s\leq l$  such that $x \in \partial\mathscr{A}(V_{g_i})$, then by the convexity of the complement components of the hypersurface amoeba $\mathscr{A}(V_{g_i})$, we can find a vector direction $w$ such that $(x+ \nu w)$ is outside $\mathscr{A}(V_{g_i})$ for all small positive numbers $\nu$. Hence,  $(x +\nu w)$ is outside $\mathscr{A}(V)$ for all small positive number $\nu$. This is in contradiction with our hypotheses.
Let $d_i$ be the distance between $x$ and the boundary of $\mathscr{A}(V_{g_i})$, and $\rho =\min\{ d_i\}_{i=1}^l$. It is claire that the ball of center $x$ and radius $\rho$ is contained in the amoeba $\mathscr{A}(V)$ (because it is contained in all the hypersurface amoebas $\mathscr{A}(V_{g_i})$). This contradict the fact that $x$ is in the boundary of the amoeba $\mathscr{A}(V)$. Hence, there exists a vector direction $v$ such that $(x+\varepsilon v)\notin \mathscr{A}(V)$ for all small  positive real numbers $\varepsilon$. We can remark that for any point $x\in\partial\mathscr{A}(V)$ there exists an open subset of unit vector  directions for which the property of {\it Claim} A is true. Indeed,  if there exists a vector direction $v$ such that $(x+\varepsilon v)\notin \mathscr{A}(V)$ for all small  positive real numbers $\varepsilon$,  and as our amoeba is the intersection of a finite number of hypersurface amoeba, then there exists $s$ with $1\leq s\leq l$ such that $x\in \partial \mathscr{A}(V_{g_{s}})$. Since any complement component of the complement of $\mathscr{A}(V_{g_{s}})$ is convex, then there exists an open neighborhood $\mathcal {V}_v$ of unit vector directions of $v$ such that $(x+\varepsilon \mathcal {V}_v)\cap \mathscr{A}(V_{g_s})$ is empty for all small  positive real number $\varepsilon$. Hence, $(x+\varepsilon \mathcal {V}_v)\cap \mathscr{A}(V)$ is empty for all small  positive real numbers $\varepsilon$.


\vspace{0.3cm}

 We will use  the following definitions: 

\vspace{0.2cm}

\begin{definition}\label{k-convexity}
An analytic subset  $\mathscr{S}$ in $\mathbb{R}^n$ of codimension 1 is said to be locally convex if and only if for any point $x\in \mathscr{S}$  there exists  an open $n$-dimensional ball $B(x,\rho )\subset \mathbb{R}^n$ of radius $\rho$ and center $x$ and a connected component of $B(x,\rho )\setminus \mathscr{S}$ which is convex. As the convexity is a local property, this means that $\mathscr{S}$ is the boundary of a convex subset in $\mathbb{R}^n$.

\end{definition}

Recall that an $n$-dimensional  subset $\mathscr{V}\subset \mathbb{R}^n$ is said to be  locally convex if for any point $x\in \mathscr{V}$ there exists  an $n$-dimensional ball $B(x, \mu )$ of center $x$ and radius $\mu$ contained in $\mathscr{V}$.

\vspace{0.2cm}

A subset  $X$ of $\mathbb{R}^n$ is convex  if  for any affine  line $\pi$ in
$\mathbb{R}^n$, the intersection $X\cap\, \pi$ has at most one connected component. 
In other words, the intersection  $X\cap\,\pi$ contains all  intervals with boundary in
$X\cap\,\pi$.
If the  points of $X$ are viewed as  $0$-cycles, then the convexity of $X$  means that  if
$a$ and $b$  are two   $0$-cycles  homologous in $X$, then they are also homologous in
$X\cap\,\pi$ where $\pi$ is the line containing  $a$ and $b$.  
Write $\tilde{H}_{*}(X, \Z)$ for the reduced integral homology of a
space $X$ with integral coefficients.
This is the kernel of the map $\deg\colon H_*(X,\Z)\to H_*({\rm pt},\Z)$ induced by the
map $X\to{\rm pt}$ to a point.

\begin{definition}
 A subset  $X$ of a vector space $V$ is $k$-{\em convex} if for any affine $(k+1)$-plane
 $\pi$, the maps $\tilde{H}_k(\pi\cap X, \mathbb{Z})\rightarrow \tilde{H}_k (X,\mathbb{Z})$
 induced by the inclusions are injective.  
\end{definition}

 This global statement generalizing convexity was found by Andr\'e Henriques~\cite{H-03}.  Moreover, Henriques
 showed that the complement of the amoeba of a codimension $r$ variety is weakly  $(r-1)$-convex. Namely,  a non-negative $(r-1)$-cycle non homologue to zero in the intersection of a $r$-plane with the complement of the amoeba is also an  $(r-1)$-cycle non homologue to zero in the complement of the amoeba itself (see  Theorem 4.1 \cite{H-03}).

\vspace{0.1cm}

\begin{lemma}\label{lemmaA}
Let $V\subset (\mathbb{C}^*)^n$ be a generic algebraic variety of codimension $r$ such that   $n\leq 2(n-r)$. Assume there exists a finite number of polynomials $\{ g_i\}_{i=1}^l$ such that 
$
\mathscr{A}(V) = \bigcap_{i=1}^l \mathscr{A}(V_{g_i}).
$
Then, for any  point $x\in \partial\mathscr{A}(V)$, there exist a connected open neighborhood $U_x\subset  \partial\mathscr{A}(V)$ of $x$,  $g_s$ with $1\leq s \leq l$, and a connected component $\mathscr{C}$ of $\partial\mathscr{A}(V_{g_s})$ such that $U_x\subset \mathscr{C}$.
\end{lemma}

\begin{proof}

 As the variety $V$ is generic  and its codimension $r$ satisfies the inequality $n \geq 2r$, then $\mathscr{A}(V)$ necessarily has  a boundary of  dimension $n-1$, which is the same dimension as the boundaries of all the hypersurface amoebas $\mathscr{A}(V_{g_i})$. 
 As the point $x$ is in the boundary $\partial\mathscr{A}(V)$,  by {\it Claim} A, we know that there is a vector direction (which we can assume unit) such that $(x+\varepsilon v) \cap \mathscr{A}(V)$ is empty for all small  positive real numbers $\varepsilon$. By the remark made in the same claim, there exists $s$ with $1\leq s\leq l$ such that $x\in \partial \mathscr{A}(V_{g_{s}})$. We claim that there exists an open neighborhood $U_x\subset \partial\mathscr{A}(V)$ of $x$ such that $U_x$ is also contained in the boundary of $\mathscr{A}(V_{g_{s}})$. 
Assume on the contrary that for any open neighborhood $U_x$ of $x$ in $\partial\mathscr{A}(V)$,  the set $U_x$ is not contained in $\partial \mathscr{A}(V_{g_{s}})$ (i.e., $U_x$ intersect the interior of $\mathscr{A}(V_{g_s})$). 
Then  for any point $y$ close to $x$ and contained  in $ \mathring{\mathscr{A}}(V_{g_s})\cap \partial\mathscr{A}(V)$, there exists
an open $n$-dimensional ball $B(y, \rho_y )\subset \mathbb{R}^n$  with center $y$  such that 
$$
B(y, \rho_y )\cap \mathscr{A}(V) = B(y, \rho_y )\cap (\bigcap_{i\in \{1,\ldots ,l\}\setminus s}\mathscr{A}(V_{g_i})).
$$ 
Namely, the set $\{ g_1,\ldots ,\hat{g_s},\ldots , g_l\}$  is a local basis of $\mathscr{A}(V)$ at $y$. Now, by the same reasoning as in {\it Claim} A, there exists $u$ with $1\leq u\leq l$ and $u\ne s$ such that $y\in \partial\mathscr{A}(V_{g_u})$. Using induction on $l$ (more precisely, induction on the number of hypersurface amoebas), we  can assume that for any point  $y\in \partial\mathscr{A}(V)$ close to $x$
there exists an open neighborhood $\mathcal{U}_{y}$ of $y$ in $\partial\mathscr{A}(V)$ with $\mathcal{U}_{y}\subset \partial\mathscr{A}(V_{g_u})$ for some $u\leq l$ and $u\ne s$.
Indeed,  if this property is not satisfied  for some $y$, by the same reasoning done for the point $x$ we can drop the number of local basis functions by one until we arrive to a hypersurface amoeba. It means that there exists an open neighborhood $\mathcal{W}_x$ of $x$ in $\partial\mathscr{A}(V)$ (may be smaller than $U_x$ ) such that $\mathcal{W}_x\setminus \{ x\}$ is covered by at most $l$ open subsets where each of them is contained in a hypersurface amoeba (their number cannot exceed $l$ because of the convexity of a hypersurface amoeba complement).
As $\partial\mathscr{A}(V_{g_u})$ are convex in the sense of Definition \ref{k-convexity}, and $l$ is finite, there exists
an open $n$-dimensional ball $B(x, \nu )\subset \mathbb{R}^n$  with center $x$  and a hyperplane $\mathscr{H}_x$ containing $x$ such that $B(x, \nu )\cap \mathscr{A}(V)$ is contained in only one side of  $\mathscr{H}_x$. This contradict the fact that the complement of the amoeba $\mathscr{A}(V)$ is $(r-1)$-convex.  Hence, there exists  an open neighborhood $U_x\subset \partial\mathscr{A}(V)$ of $x$ such that $U_x$ is also contained in the boundary of $\mathscr{A}(V_{g_{s}})$ for some $s$.

\end{proof}

 \vspace{0.2cm} 

\noindent {\it Proof of Theorem \ref{Amoeba Basis} for $n \geq 2r$.} With the same notation as above, assume that $r>1$.  Then  $\partial\mathscr{A}(V)$ has only one noncompact connected component $\mathscr{C}(V)$ (i.e., unbounded connected component). Indeed, by Bergman \cite{B-71} (see also Bieri and Groves  \cite{BG-84})
the  logarithmic limit set of an algebraic variety of codimension $r$
is  $(n-r-1)$-dimensional (i.e., its dimension is strictly less than $n-2$). So, the complement of the amoeba $\mathscr{A}(V)$  in $\mathbb{R}^n$ has only one noncompact connected component. Assume $x$ is contained in the noncompact connected component  $\mathscr{C}(V)$ of the boundary of $\mathscr{A}(V)$. Lemma \ref{lemmaA} shows that $\mathscr{C}(V)$ is locally convex viewed as the graph of a function (this is a fact of the convexity of the complement components of  hypersurface amoebas).
As the convexity  is a local property, i.e., a subset of a vector space is globally convex  if and only if it is locally convex, this implies that $\mathscr{C}(V)$ is globally convex.
This contradict the fact that the logarithmic limit set of an algebraic variety of codimension $r$ is $(n-r-1)$-dimensional. Also, this contradict the higher convexity of the complement  components of the amoeba $\mathscr{A}(V)$. In fact, if a subset of $\mathbb{R}^n$ is $k$-convex with nontrivial $k$-homology, then it can never be convex. Namely, we know that the homology of degree $(r-1)$ of the unbounded complement of the amoeba $\mathscr{A}(V)$ is nontrivial. This is a consequence of the injection of the $(r-1)$-homology of the complement of its logarithmic set in the sphere $S^{n-1}$ into the $(r-1)$-homology of the complement of the amoeba.

\vspace{0.2cm}

 \newpage
 
 \begin{lemma}\label{lemmaB}
Let $V\subset (\mathbb{C}^*)^n$ be a generic algebraic variety of codimension $r$ such that   $n > 2(n-r)$. Assume there exists a finite number of polynomials $\{ g_i\}_{i=1}^l$ such that 
$
\mathscr{A}(V) = \bigcap_{i=1}^l\mathscr{A}(V_{g_i}).
$
Then,  the complement of the  amoeba $\mathscr{A}(V)$ contains a component which is not $(r-1)$-convex.  
\end{lemma}

\begin{proof}
 In this case, the amoeba  $\mathscr{A}(V)$   may have or may not have  a  boundary. The variety $V$ is generic,  and its codimension $r$ satisfies the inequality $n < 2r$, means that  $ \mathscr{A}(V)$  is $2(n-r)$-dimensional  i.e., its dimension is strictly less than the dimension of the ambient space $\mathbb{R}^n$.
 Let $x$ be a  point in $\mathscr{A}(V)$, and for simplicity assume that $x$ is a smooth point of the amoeba. The amoeba   $\mathscr{A}(V)$ is $2(n-r)$-dimensional, and $2(n-r) < n$ implies that for a small open neighborhood $U_x$ of $x$ in $\mathscr{A}(V)$ there exists a  unit vector direction $v$ not in the tangent space of $\mathscr{A}(V)$ at $x$ such that $(U_x+\varepsilon v)\cap \mathscr{A}(V)$ is empty for all small positive numbers $\varepsilon$. The same reasoning as in Lemma \ref{lemmaA}  shows that there exists $s$ with $1\leq s\leq l$ such that $(U_x+\varepsilon v)\cap \mathscr{A}(V_{g_s})$ is empty for all small positive numbers $\varepsilon$. Hence, $U_x$ is contained in the boundary of the hypersurface amoeba $\mathscr{A}(V_{g_s})$. As the  $\partial\mathscr{A}(V_{g_s})$ is locally convex in the sense of Definition \ref{k-convexity}, then there exists a hyperplane
   $\mathscr{H}_x\subset \mathbb{R}^n$  passing throughout  the  point $x$ such that for a small ball $ B(x,\rho )\subset \mathbb{R}^n$ the set   $\partial\mathscr{A}(V_{g_s})\cap B(x,\rho )$ is contained in only one side of $\mathscr{H}_x$.
   Let $\mathscr{L}_x\subset \mathscr{H}_x$ be an $r$-dimensional plane containing $x$. Let $\mathscr{A}^c(V) := \mathbb{R}^n\setminus \mathscr{A}(V)$ i.e.,  the complement of the amoeba in $\mathbb{R}^n$. 
Now it is claire that there exists an $(r-1)$-cycle $\gamma$ in $\mathscr{L}_x$ (which we can assume positive in the sense of Henriques's Definition 3.3 in  \cite{H-03}) non homologue to zero in $\mathscr{L}_x\cap \mathscr{A}^c(V)$. In fact, take a small  $(r-1)$-dimensional sphere in $\mathscr{L}_x$ centered at $x$. As $r<n$  (i.e., $V$ is not a set of points),
the cycle $\gamma$ bounds an $r$-chain  in $\mathscr{A}^c(V)$ (because a small neighborhood of $x$ in the amoeba $\mathscr{A}(V)$ is  contained in only one side of the hyperplane $\mathscr{H}_x$). This means that the homology  class of $\gamma$ in $H_{r-1}(\mathscr{A}^c(V))$ is trivial. This contradict the higher convexity of the complement of the amoeba $\mathscr{A}(V)$ (see \cite{H-03}).
 
\end{proof} 
 
 \vspace{0.2cm} 

\noindent   If the codimension  $r$ of the variety  $V$ satisfies $n> 2(n-r)$, then Theorem \ref{Amoeba Basis} is a consequence of Lemma \ref{lemmaB}. 

\begin{remark} ${}$ 

\begin{itemize}
\item[(a)]\,
Let  $\{ g_j\}_{j=1}^r$  be  a generator set   of the defining  ideal $\mathcal{I}(V)$ of an algebraic variety $V$. As a consequence of the  proof of Theorem 1.1 in \cite{NP-11}, the degree of the  polynomials $f\in \mathcal{I}(V)$ such that $\mathscr{A}(V) =\bigcap_{f\in\mathcal{I}(V)}\mathscr{A}(V_{f})$ can always be bounded by $2 \max_{j=1}^r \{ \deg (g_j)\}$.
\item[(b)]\,  Using the higher convexity of coamoeba complements proved by Sottile and I in \cite{NS-13}, the statement  of Theorem \ref{Amoeba Basis} is valid if we replace amoebas by coamoebas.
\end{itemize}
 \end{remark}
 
\section{Example of a generic affine line in $(\mathbb{C}^*)^3$}

The amoeba of a generic line $L$ in $(\mathbb{C}^*)^3$ is a surface with Figure 1 or without boundary Figure 2 (it depends if it is real or not real, see \cite{NP-11} for more details).  By Lemma \ref{lemmaB}, if the amoeba $\mathscr{A}(L)$ of $L$ is the intersection of a finite  number of hypersurface amoeba, then it is locally a
  convex surface because it is locally contained in the boundary of a hypersurface amoeba  $\mathscr{A}(V_{g_s})$. Namely, if $x$ is a point in $\mathscr{A}(L)$, then there exists  $1\leq s\leq l$ and  an open neighborhood $U_x$ of $x$ in $\mathscr{A}(L)$ which is contained  in  $\partial\mathscr{A}(V_{g_s})$.   Hence,   the  convexity of the complement of the  hypersurface  amoeba $\mathscr{A}(V_{g_s})$ implies that   there exists a 1-cycle $\gamma\subset \mathscr{H}_x$ such that the class of $\gamma$ in $H_1(\mathscr{A}(L)^c\cap \mathscr{H}_x , \mathbb{Z})$ is different than zero, where $\mathscr{H}_x$ is a hyperplane in $\mathbb{R}^3$ which separate locally at $x$ the  boundary  $\partial\mathscr{A}(V_{g_s})$ of the amoeba $\mathscr{A}(V_{g_s})$ (i.e., locally in a small neighborhood of $x$, the boundary  $\partial\mathscr{A}(V_{g_s})$ is situated in only one side of $\mathscr{H}_x$).
  But   $\gamma$ bounds a topological disk in $\mathscr{A}(L)^c$. This contradict the 1-convexity of   $\mathscr{A}(L)^c$ in $\mathbb{R}^3$. Hence, the amoeba of a generic line in the space can never be the intersection of a finite number of hypersurface amoebas.


\begin{figure}[h!]
\begin{center}
\includegraphics[angle=0,width=0.74\textwidth]{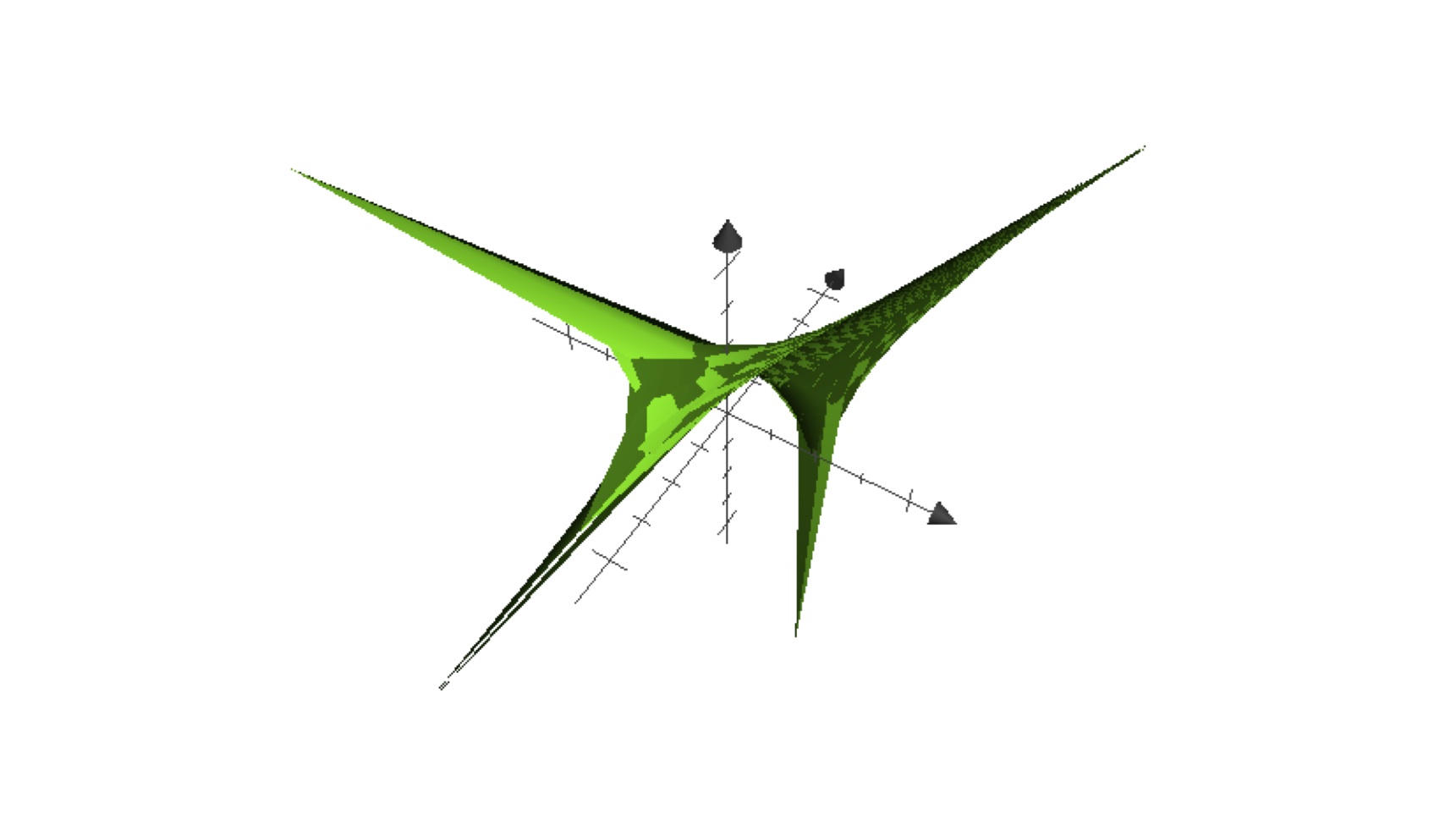}
\caption{The amoeba of the real line in $(\mathbb{C}^*)^3$ given by the parametrization 
$\rho (z)=(z,z+\frac{1}{2},z-\frac{3}{2})$.  In this case, the amoeba is topologically the closed disk without four points of its boundary.}
\label{c}
\end{center}
\end{figure}

\begin{figure}[h!]
\begin{center}
\includegraphics[angle=0,width=0.7\textwidth]{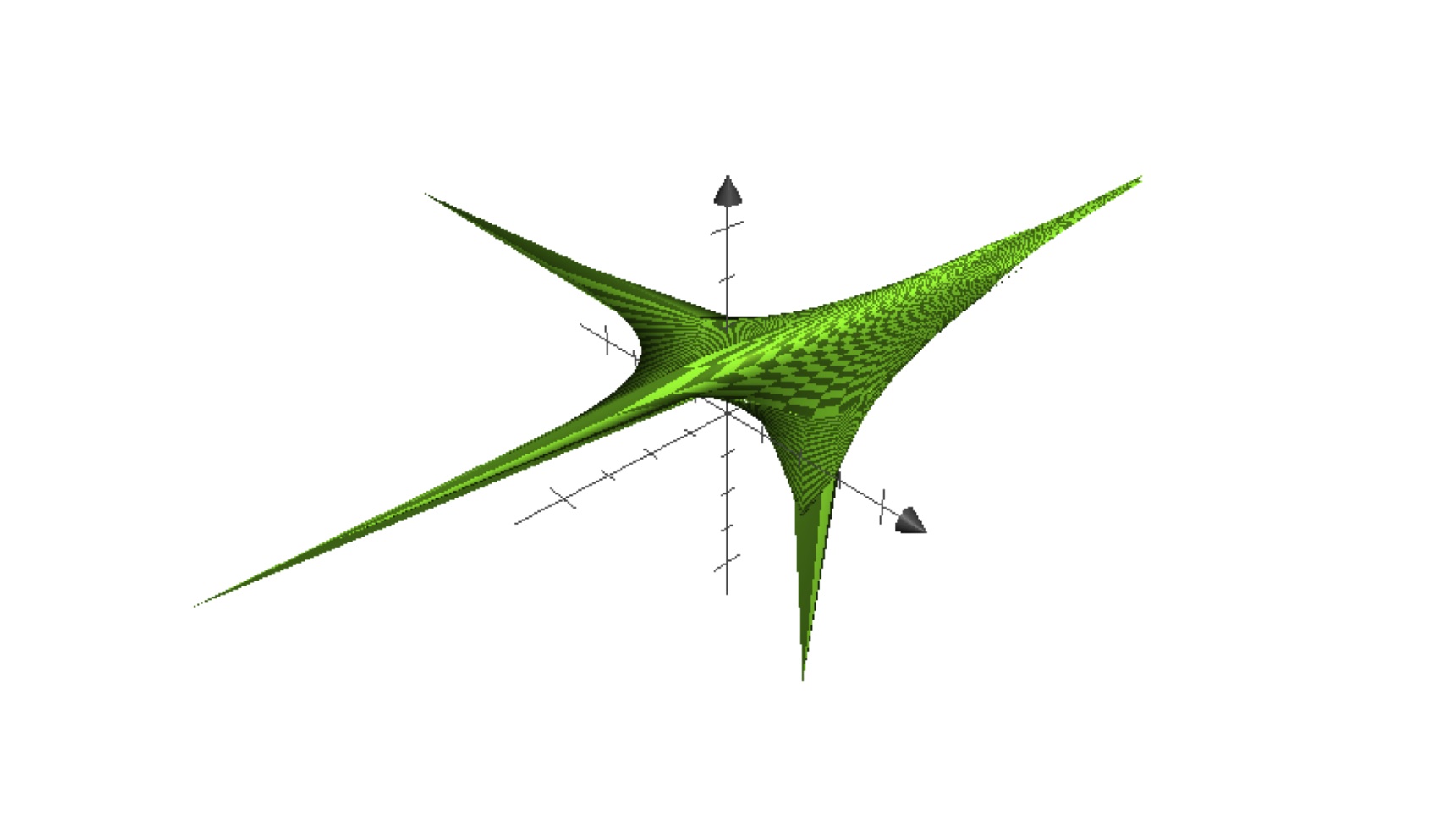}
\caption{the amoeba of non real line in $(\mathbb{C}^*)^3$ given by the parametrization $\rho (z)=(z,z+1,z-2i)$. In this case, the amoeba is topologically  the Riemann sphere without four points.}
\label{c}
\end{center}
\end{figure}


\section{Characterization of  hypersurface amoeba boundaries}

Given a smooth algebraic hypersurface $V\subset (\mathbb{C}^*)^n$,  F. Schroeter and T. de Wolff  give a characterization of the boundary of hypersurface amoeba $\mathscr{A}(V)$ up to singular points of the set of critical values of the logarithmic map restricted to $V$ (see Theorem 1.3 in \cite{SW-13}).  In this section, we will refine their theorem and give a  very short proof of it. Let us start by giving some definitions and notation. 

\vspace{0.2cm}

\noindent Let us denote by $\Log$ the coordinatewise logarithmic map, i.e., the map from the complex algebraic torus into $\mathbb{R}^n$ defined as follows:
\begin{eqnarray*}
\Log :&(\mathbb{C}^*)^n&\longrightarrow\,\,\, \mathbb{R}^n\\
&(z_1,\ldots ,z_n)&\longmapsto\,\,\, (\log |z_1|,\ldots , \log |z_n|).
\end{eqnarray*}

\vspace{0.2cm}

\noindent We denote by $Critp(\Log_{|V})$ (resp. $Critv(\Log_{|V})$) the set of critical points (resp. critical values) of the logarithmic map restricted to $V$. 

\vspace{0.2cm}

Let  $V\subset (\mathbb{C}^*)^n$ be  a complex algebraic hypersurface
defined by a polynomial $f$ and nowhere singular. The {\em logarithmic Gauss map} of the hypersurface $V$
is a rational map from  all $V$) to $\mathbb{CP}^{n-1}$ defined as follows:
 \begin{eqnarray*}
\gamma :&V&\longrightarrow\,\,\,\mathbb{CP}^{n-1}\\
&z&\longmapsto\,\,\, \gamma (z) =  [z_1\frac{\partial f}{\partial z_1}(z):\cdots :z_n\frac{\partial f}{\partial z_n}(z)],
\end{eqnarray*}

\noindent We have the following  commutative diagram:

\begin{equation}
\xymatrix{
V \ar[rr]^{\gamma}&& \mathbb{C}\mathbb{P}^{n-1}\cr
Critp(\Log_{|V})\ar[rr]^{\gamma_c}\ar[u]^{\cup}\ar[dr]_{\Log}&&\mathbb{R}\mathbb{P}^{n-1}\ar[u]_{\cup} \cr
&Critv(\Log_{|V})\ar[ru]_{g},
}\nonumber
\end{equation}
where $\cup$  denotes the natural inclusion,
$g$ is the usual Gauss map defined on the smooth part of 
$Critv(\Log_{|V})$,  and $\gamma_c = \gamma_{|Critp(\Log_{|V})}$ (i.e., the restriction of $\gamma$ to the set of  critical points $Critp(\Log_{|V})$ of the logarithmic map).

\begin{definition}\label{regular-point}
A point  $x$  in $Critv(\Log_{|V})$ is called regular if and only if  $\Log^{-1}(x)\cap V$ is contained in the set of regular points of the restriction of the logarithmic Gauss map to $Critp(\Log_{|V})$ (i.e., $\Log^{-1}(x)\cap V$ contains no critical point of $\gamma$).
\end{definition}

\begin{remark} ${}$
\begin{itemize}
\item[(i)]\,  A singular point of $Critv(\Log_{|V})$   can be a regular point in the sense of  Definition \ref{regular-point}. Moreover,  a non regular point  in the sense of Definition \ref{regular-point} is necessarily  a singular point of $Critv(\Log_{|V})$;
\item[(ii)]\,   The degree of the extension  of $\gamma$ to the compactification $\overline{V}$ of $V$ in the projective space $\mathbb{C}\mathbb{P}^{n}$ is equal to $n!\Vol (\Delta_f)$, i.e., the cardinality of $\gamma^{-1}(y)$  for a generic point $y$ is finite and equal to $n!\Vol (\Delta_f)$ (see \cite{M-00});
 \item[(iii)]\,  The  inverse image of a regular point $x\in Critv(\Log_{|V})$ by the logarithmic map is a finite number of points. But  the inverse image by the logarithmic map of a non regular point  can be of positive dimension (see the example of the hyperbola in Section 5). 
\item[(iv)]\,  In the case of plane curves, the logarithmic Gauss map $\overline{V}\rightarrow \mathbb{C}\mathbb{P}^{1}$
 is a branched covering where the  branching points are the points of
logarithmic inflection (in other words, inflection after taking the holomorphic logarithm); see \cite{M-00}.
\end{itemize}
\end{remark}

 \noindent Let  $x \in Critv(\Log_{|V})$ be a regular point  contained   in the image by the logarithmic map of the local holomorphic branches   $\mathscr{B}_1(x), \ldots , \mathscr{B}_s(x)$, and we denote by $\mathscr{C}_1(x), \ldots , \mathscr{C}_s(x)$  their corresponding  real branches of  critical values passing through  $x$.
Recall that  the $\mathscr{C}_i(x)$'s  are the image of the critical points inside the corresponding holomorphic branches  and    $\mathscr{C}_i(x)$ can be empty if the local branch $\mathscr{B}_i(x)$ is regular (i.e., does not intersect $Critp(\Log_{|V})$).
In  general, if $V$ is not smooth, it can happen that 
    a critical value branch $\mathscr{C}_i(x)$  has  dimension strictly less than $(n-1)$   and  then  $x$ is necessarily contained in the interior of the amoeba. So, throughout this section,  we assume that the dimension of $\mathscr{C}_i(x)$  is equal to $(n-1)$ for all $i$ (i.e.,  we can assume $V$ singular but we consider only the set of smooth points of $V$).  We denote by $v_i(x)$ the normal vector to $\mathscr{C}_i(x)$ (if it is nonempty) pointed inside the local amoeba $\mathscr{A}(\mathscr{B}_i(x))$ of $\mathscr{B}_i(x)$ (the existence of $v_i(x)$ is assured by the regularity of $x$). We have the following:

\begin{lemma}\label{lemmaC}
Let $V$ be a complex algebraic hypersurface. Let $x$ be a point in the boundary of the amoeba $\partial\mathscr{A}(V)$. then the set $(\Log^{-1}(x)\cap V)$ is contained in $Critp(\Log_{|V})$.
\end{lemma}

\begin{proof}
Assume there exists a component $\mathcal{C}_x$ of $\Log^{-1}(x)\cap V$ which is not critical. This means that there exists $z\in V$ such that $\Log (z) = x$ and  $z$ is a regular point of the logarithmic map (i.e., the Jacobian $\Jac (\Log_{|V})_{z}$
of the logarithm map restricted to $V$ at the point $z$ has maximal rank).
The fact that the set of regular points of the logarithmic map is an open subset of $V$, implies that there exists an open subset $U_z$ in $V$ containing $z$, such that $\Log_{|U_z}$ is a submersion. Hence, the point 
$x$ must be in the interior of the amoeba 
and  not in its boundary. This contradict our  hypothesis on $x$.
\end{proof}

\begin{lemma}\label{lemmad}
Let $V$ be a complex algebraic hypersurface. Let $x$ be a regular  point of $Critv(\Log_{|V})$. Then with the above notation, the following statements are equivalent:
\begin{itemize}
\item[(i)]\, The point $x$ is  in $\partial\mathscr{A}(V)$;
 \item[(ii)]\, The convex hull of the vectors $\{v_i(x)\}_{i=1}^s$ does not contain the origin and
 the intersection of each holomorphic branch $\mathscr{B}_i(x)$ with $\Log^{-1}(x)$ is contained in
 $Critp(\Log_{|V})$.
 \end{itemize}
\end{lemma}

\begin{proof}
 $(i)\Longrightarrow (ii)$. As $x\in \partial\mathscr{A}(V)$, then there exists a vector direction  $v$ such that $(x+ \varepsilon v)\notin \mathscr{A}(V)$ for all small strictly positive numbers $\varepsilon$ (see {\it Claim} A). This is equivalent to the fact  that  there exists $\eta$ such that for  any strictly  positive number $\varepsilon \leq \eta$ the
 vector $ (x+\varepsilon v)$   is outside the convex hull of $\{ x, (x+ v_i(x))_{i=1}^s\}$. Indeed, if there exists a sequence $\varepsilon_m$ such that $ (x+\varepsilon_m v)$ is contained in the convex hull of $\{ x, (x+ v_i(x))_{i=1}^s\}$, then the fact that the number of branches is finite (because $V$ is algebraic), implies that for a small $\varepsilon$, the vector $(x+\varepsilon v)$ is contained in the local amoeba of some local holomorphic branch  $\mathscr{B}_s(x)$, and then $(x+\varepsilon v)$ is contained in the amoeba itself. This contradict the choice of $v$, and then $x$ must be a vertex of the convex hull of $\{ x, (x+ v_i(x))_{i=1}^s\}$. This means that the convex hull of the vectors $\{v_i(x)\}_{i=1}^s$ does not contain the origin.
 Finally, 
  the fact that all the local holomorphic branches $\mathscr{B}_i(x)$ intersect 
  $Critp(\Log_{|V})$ with $x\in \mathscr{C}_i(x)$ is a consequence of Lemma \ref{lemmaC}.\\
 $(ii)\Longrightarrow (i)$. All the local holomorphic branches $\mathscr{B}_i(x)$ intersect $Critp(\Log_{|V})$ with $x\in \mathscr{C}_i(x)$   and the convex hull of the vectors $\{v_i(x)\}_{i=1}^s$ does not contain the origin means that  $x$ is not  in the convex hull of the vectors $\{(x+v_i(x))\}_{i=1}^s$. This implies that
 there exists a vector direction $v$ not in the convex hull of $\{v_i(x)\}_{i=1}^s$ such that $(x+\varepsilon v)\notin \mathscr{A}(V)$ for any small positive number $\varepsilon$.  In fact, if for any vector direction $v$, the vector $(x+\varepsilon v)$ is contained in $\mathscr{A}(V)$ means that for any $v$ there exists a local holomorphic branch $\mathscr{B}_s(x)$ such that $(x+\varepsilon v)$ is contained in the local amoeba of $\mathscr{B}_s(x)$ for any small strictly positive number $\varepsilon$. As the number of local holomorphic branches is finite, then the point $x$ must be in the interior of the convex hull of $\{(x+v_i(x))\}_{i=1}^s$. This contradict our hypothesis and then the point $x$ is in the boundary of the amoeba of $V$. 
\end{proof}

\begin{theorem}\label{boundary Amoeba}
Let $V$ be a complex algebraic hypersurface and $x$ be a regular point in $Critv(\Log_{|V})$ with $(\Log^{-1}(x)\cap V) \subset Critp(\Log_{|V})$. Then the convex hull of the vectors $\{v_i(x)\}_{i=1}^s$ does not contain the origin
if and only if $x\in \partial\mathscr{A}(V)$. In other words, the convex hull of the vectors $\{v_i(x)\}_{i=1}^s$  contains the origin if and only if  the point $x$ is contained in the interior of the amoeba.
\end{theorem}

\begin{proof}
Let $x$ be a  point of $Critv(\Log_{|V})$ such that  the set $(\Log^{-1}(x)\cap V) \subset$\\ $
 Critp(\Log_{|V})$ and suppose the convex hull of $\{(x+ v_i(x))\}_{i=1}^s$ contains $x$. This means that for any unit vector direction $v$ and for all small positive numbers $\varepsilon$ we have $(x+\varepsilon v) \in \mathring{\mathscr{A}}(V)$, which is equivalent to the fact that $x$ is contained in the interior of the amoeba $\mathscr{A}(V)$. If the convex hull of the vectors $\{v_i(x)\}_{i=1}^s$ does not contain the origin, by Lemma \ref{lemmad} and the hypothesis of our theorem,  the point $x$ is in the boundary of the amoeba $\partial\mathscr{A}(V)$.

\end{proof}

\section{Example of a non regular point in $Critv(\Log_{|V})$}

Let $\mathcal{H}$ be the real algebraic plane curve (hyperbola) parametrized as follows:
\begin{eqnarray*}
\rho :&\mathbb{C}^*\setminus \{ -1, -\frac{1}{6}\}&\longrightarrow\,\,\,(\mathbb{C}^*)^{2}\\
&z&\longmapsto\,\,\, \rho (z) =  -\frac {z+\frac{1}{6}}{z+1}.
\end{eqnarray*}
	Its amoeba has a non regular critical value $x_0$, called a pinching point by Mikhalkin (Remark 10,  \cite{M-00}). The inverse image of the point $x_0 = (-\frac{\log 3}{2}, \log |\frac{\sqrt{3} -5}{8}|)$  by the logarithmic map in $\mathcal{H}$ is a non geodesic circle (i.e.,  $\mathcal{H}\cap\Log^{-1}(x_0)$ is a circle but not geodesic in the flat torus $(S^1)^2 =\Log^{-1}(x_0)$).  As the set of critical points of the logarithmic map and the argument maps coincides (see \cite{M-00}),
we can check the  fact that $\mathcal{H}\cap\Log^{-1}(x_0)$ is a circle
by looking to the coamoeba of $\mathcal{H}$. The set of critical values of the argument map is a non geodesic circle $\mathscr{C}$ which has two different real points (i.e., intersect the finite real subgroup $(\mathbb{Z}_2)^2$ of  the hall real torus in two points) union the isolated point $(\pi , \pi)$. The circle $\mathscr{C}$ is also critical for the logarithmic Gauss map $\gamma_c$. More precisely, there are two real branches $\mathscr{B}_1$ and $\mathscr{B}_2$ of critical points intersecting $\mathscr{C}$ in two different real points contained in two different quadrants of $(\mathbb{R}^*)^2$, namely the quadrants $(+,-)$ and $(-,+)$. The image of each branch by the logarithmic map has an inflection point at $x_0$.

\begin{figure}[h!]
\begin{center}
\includegraphics[angle=0,width=0.76\textwidth]{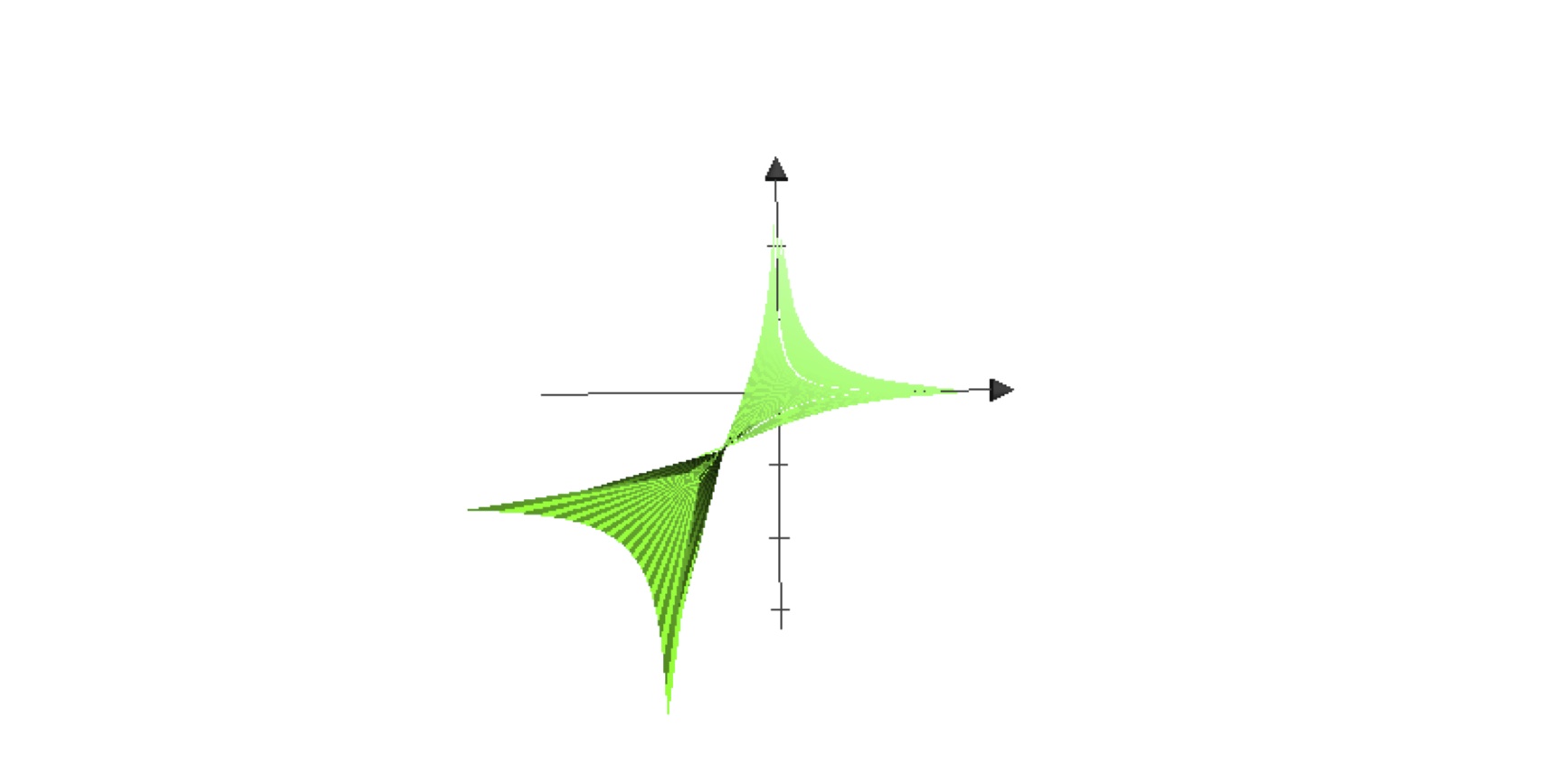}\quad
\includegraphics[angle=0,width=0.2\textwidth]{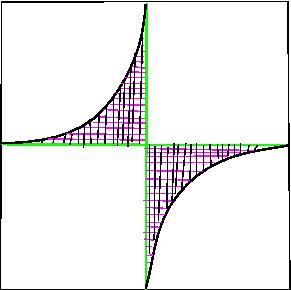}
\caption{the amoeba  and the coamoeba of the real hyperbola in $(\mathbb{C}^*)^2$ with defining polynomial $f(z,w) = \frac{1}{6} +z+w+zw$.}
\label{c}
\end{center}
\end{figure}


\end{document}